\documentclass[12pt]{amsart}

\usepackage{amsmath,amssymb,fullpage,xy,mathrsfs,amsthm,xcolor,amsxtra,graphicx,tikz}
\usepackage{hyperref}
\hypersetup{colorlinks=true,urlcolor=blue,citecolor=blue,linkcolor=blue}
\usepackage{enumerate}
\usepackage{colonequals}

\usepackage[mathcal]{euscript}
\newcommand{\Gal}{\operatorname{Gal}}

\newcommand{\p}{\mathfrak{p}}

\newcommand{\Z}{\mathbf{Z}}

\newcommand{\NN}{\mathbf{N}}
\newcommand{\R}{\operatorname{\mathbf{R}}}

\newcommand{\Q}{\mathbf{Q}}

 \newcommand{\NP}{\mathrm{NP}}

\newcommand{\ddef}{\colonequals}
\newcommand{\disc}{\operatorname{disc}}

\numberwithin{equation}{subsection}

\theoremstyle{plain}
\newtheorem{thm}[equation]{Theorem}

\newtheorem{lem}[equation]{Lemma}
\newtheorem{defn}[equation]{Definition}
\newtheorem{cor}[equation]{Corollary}
\newtheorem{prop}[equation]{Proposition}

\theoremstyle{remark}
\newtheorem{rmk}[equation]{Remark}

 \input xy
 \xyoption{all}

\begin{document}

\title{On the arithmetic of Pad\'e approximants to the exponential function}

\author{John Cullinan}
\address{Department of Mathematics, Bard College, Annandale-On-Hudson, NY 12504, USA}
\email{cullinan@bard.edu}
\urladdr{\url{http://faculty.bard.edu/cullinan/}}


\author{Nick Scheel}
\address{Department of Mathematics, Bard College, Annandale-On-Hudson, NY 12504, USA}
\email{ns5044@bard.edu}

\begin{abstract}
The $(u,v)$-Pad\'e approximation to a function $f$ is the (unique, up to scaling) rational approximation $f(x) = P(x)/Q(x) + O(x^{u+v+1})$, where $P$ has degree $u$ and $Q$ has degree $v$.  Motivated by recent  work of Molin, Pazuki, and Rabarison, we study the arithmetic of the Pad\'e approximants of the  exponential polynomials.  By viewing the approximants as certain Generalized Laguerre Polynomials, we determine the Galois groups of the diagonal approximants and prove some special cases of irreducibility.
\end{abstract}

\keywords{Galois Theory, Orthogonal Polynomials, Irreducibility}

\maketitle


\section{Introduction} \label{intro}

In this short paper we take up a question raised in \cite{pade} on the arithmetic of the Pad\'e approximants to a given function.  We immediately restrict ourselves to approximants of the truncated exponential polynomials in the hope that our methods prove useful to other, well-known examples.  Looking further, one might reasonably ask for a general relationship between the irreducibility and Galois groups of an irreducible polynomial and its Pad\'e approximants.

We begin by recalling the notion of the Pad\'e approximation.  Let $f \in \Q[x]$ be a polynomial of degree $n \geq 1$, let $\NN = \Z_{\geq 1}$, $(u,v) \in \NN \times \NN$, and recall that the $(u,v)$-Pad\'e approximation to $f$ is the (unique, up to scaling) rational function $R \in \Q(x)$ such that 
$$
f(x) = R(x) = P(x)/Q(x) + O(x^{u+v+1}),
$$
where $P(x)$ has degree $\leq u$, $Q(x)$ has degree $\leq v$ and $u+v \leq n$.  

The normalized, truncated exponential polynomial $e_n(x)$ is given by
$$
e_n(x) = n!\sum_{j=0}^n \frac{x^j}{j!},
$$
and, for $(u,v) \in \NN \times \NN$, has well known, normalized, Pad\'e approximants 
\begin{align} \label{first_pade}
P(u,v,x) &= \sum_{j=0}^u \frac{(u+v-j)!}{v!} \binom{u}{j} x^j, \text{and} \\
Q(u,v,x) &= \sum_{j=0}^v \frac{(u+v-j)!}{u!} \binom{v}{j} (-x)^j;
\end{align}
see \cite{iserles} for a statement and derivation. In particular, we observe that 
\begin{align} \label{identity}
P(u,v,x) = Q(v,u,-x).
\end{align}
In this paper we study the families $P(u,v,x)$ and $Q(u,v,x)$ for their Galois groups and their irreducibility.  Because of the identity (\ref{identity}), it will suffice to work with the family of $P(u,v,x)$.

To the best of our knowledge, the study of the arithmetic of Pad\'e approximants was first proposed in the recent preprint \cite{pade}, with the exponential polynomials as one of the guiding examples.  Because of the beautiful arithmetic properties of the exponential polynomial (see \cite{coleman} for an overview), it makes sense to use them as a test case.  Other families of polynomials -- the Legendre, Laguerre, and Jacobi, for example -- have conjecturally-similar properties to the exponential, but global results about their irreducibility and Galois groups are not as precisely known.

We can reframe this entire setup in terms of the Generalized Laguerre Polynomials and appeal to known results to sort out what is left to prove.  We begin by recalling the definition of the Generalized Laguerre Polynomials (GLP):
\begin{align} \label{glp}
L_n^{(\alpha)}(x) \ddef (-1)^n\sum_{j=0}^n \binom{n+\alpha}{n-j} \frac{(-x)^j}{j!}.
\end{align}
We adopt the linear change of variables and scaling introduced in \cite[(1.4)]{hajir2}:
\begin{align} \label{LL}
\mathcal{L}_n^{\langle r \rangle}(x) \ddef n! L_n^{(-1-n-r)}(x) = \sum_{j=0}^n \binom{n}{j} (r+1)(r+2)\cdots (r+n-j) x^j.
\end{align}
Viewed in this way, it is easy to verify that $e_n(x) = \mathcal{L}_n^{\langle 0 \rangle}(x)$. Moreover, by comparing the coefficients of (\ref{first_pade}) to (\ref{LL}), we see further that
\begin{align*}
P(u,v,x) &= \mathcal{L}_u^{\langle v \rangle}(x), \text{ and} \\
Q(u,v,x) &= \mathcal{L}_v^{\langle u \rangle}(-x).
\end{align*}
Therefore, the $(u,v)$-Pad\'e approximation to $e_n(x)$ is simply the identity
\begin{align} \label{pade_exp}
\mathcal{L}_n^{\langle 0 \rangle}(x) = \frac{\mathcal{L}_u^{\langle v \rangle}(x)}{\mathcal{L}_v^{\langle u \rangle}(-x)} + O(x^{u+v+1}),
\end{align}
up to scaling by a rational number. Therefore, all questions about the arithmetic of the Pad\'e approximants to $e_n(x)$ can be reframed as questions about the GLP.  

The main questions that we are interested in center around irreducibility and Galois groups.  With that in mind, there has already been a good deal of work done surrounding these properties, both for the exponential polynomials and for the GLP.  For the exponential polynomials, we refer to \cite{coleman} for proofs of the irreducibility of $e_n(x)$ for all $n\geq 1$ and for the fact that the Galois group always contains $A_n$, and equals $A_n$ if and only if $n\equiv 0\pmod{4}$.  

For the GLP, we refer to \cite{fl, hajir, hajir2} for the proofs of the fact that if $r$ is fixed, then for all but finitely many $n$, the $\mathcal{L}_n^{\langle r \rangle}(x)$ are irreducible with Galois group containing $A_n$.  There are additionally several specific cases where we know the Galois group for all $n$.  In particular, if $0 \leq r \leq 8$, then, by \cite[Theorem 1.3]{hajir}, the $\mathcal{L}_n^{\langle r\rangle}(x)$ are irreducible for all $n \geq 1$ with Galois group containing $A_n$.  

Of particular relevance to us is the case where $r=n$.  These are the Bessel polynomials and it is known by work of Fileseta-Trifonov \cite{FT} that they are irreducible for all $n\geq 1$ and by work of Grosswald \cite{grosswald} that they have Galois group $S_n$ for all $n\geq 1$.  

This brings us to the main purpose of this note.  The authors in \cite{pade} raise the question of the \emph{diagonal approximants} to the exponential polynomial.  We recall their definition and give a short recap of their associated experimental results.  

\begin{defn}
Let $n$ be a positive integer $\geq 2$ and write $n=2m+1-\delta$, for some $m \in \Z_{\geq 1}$ and $\delta \in \lbrace 0,1\rbrace$. The \emph{diagonal Pad\'e approximants} to the truncated exponential polynomial $e_{n}(x)$ are $P(m,m+\delta,x)$ and $Q(m,m+\delta,x)$.
\end{defn}

In \cite[\S4.2.1]{pade} they present some computational evidence for the Pad\'e approximants to have large Galois group.  Rewriting their table in our notation, they present us with the following data:

\begin{center}
\begin{tabular}{|c|c|c||c|c|c|}
\hline
$m$ & $\Gal_\Q P(m,m,x)$ & $\Gal_\Q Q(m,m,x)$ &  $m$ & $\Gal_\Q P(m,m+1,x)$ & $\Gal_\Q Q(m,m+1,x)$\\
\hline
6 & $S_6$ & $S_6$ & 4 & $A_4$ & $S_5$ \\
8 & $S_8$ & $S_8$ & 5 & $A_8$ & $A_9$\\
9 & $S_9$ & $S_9$ & 12 & $A_{12}$ & $S_{13}$\\
20 & $S_{20}$ & $S_{20}$ & 16 & $A_{16}$ & $S_{17}$\\
&&& 19 & $S_{19}$ & $S_{20}$\\
&&& 20 & $A_{20}$ & $S_{21}$\\
\hline
\end{tabular}
\end{center}
In light of the results of Grosswald \cite{grosswald}, the first two columns of this table continue indefinitely: $\Gal_\Q(P(m,m,x)) = \Gal_\Q(Q(m,m,x)) \simeq S_m$ for all $m \geq 1$.  We therefore turn our attention to the remaining diagonal approximants, which is the case where $n=\deg e_n(x)$ is even.  The Galois theory of these polynomials is not covered by \cite{grosswald}, nor is the irreducibility covered by \cite{FT}.  Our first result is as follows.

\begin{thm} \label{diagthm}
Let $m \in \Z_{\geq1}$ and assume the polynomials $P(m,m+1,x)$ and $Q(m,m+1,x)$ are irreducible over $\Q$.  Then $\Gal_{\Q} P(m,m+1,x) \supseteq A_{m}$, and $\Gal_{\Q} Q(m,m+1,x) \supseteq A_{m+1}$.  Furthermore,
\begin{enumerate}
\item $\Gal_{\Q} P(m,m+1,x) \simeq A_{m}$ if and only if $m\equiv 0\pmod{4}$ or if $m=2(2k+1)^2-1$ for some $k\geq0$, and
\item $\Gal_{\Q} Q(m,m+1,x) \simeq A_{m+1}$ if and only if $m=(2k+1)^2-1$ for some $k\geq1$.
\end{enumerate}
\end{thm}

Determining the general irreducibility of these polynomials is a much more difficult problem, and one that we do not take on in this paper.  However, there are some special cases in which the $P(m,m+1,x)$ and $Q(m,m+1,x)$ are Eisenstein (more precisely, satisfy the Eisenstein-Dumas criterion).

\begin{thm}\label{edthm}
Let $p$ be an odd prime and let $n \geq 1$. Then the polynomials $P(p^n,p^n+1,x)$, $Q(p^n,p^n+1,x)$, $P(p^n,p^n-1,x)$, $Q(p^n-1,p^n,x)$ are irreducible over  $\Q$.
\end{thm}

One of our goals in writing this paper is to frame the Pad\'e approximation in terms of other well-known families of polynomials.  Slightly more broadly, we can rewrite (\ref{pade_exp}) in terms of the hypergeometric functions (up to scaling):
$$
_1F_1(-u-v-\delta,-u-v-\delta;x)={\frac  {{}_{1}F_{1}(-u;-u-v;x)}{{}_{1}F_{1}(-v;-u-v;-x)}} + O(x^{u+v+1}),
$$
and ask whether similar relationships that hold for more general hypergeometric functions will aid in determining their irreducibility and Galois properties.   \\

\noindent \textbf{Notation.} If $p$ is a prime number then we write $v_p: \Q \to \Z$ for the $p$-adic valuation and $\Q_p$ for the $p$-adic numbers.  Following the conventions of \cite{coleman}, let $g \in \Q_p[x]$ and write 
$$
g(x) = a_0 + a_1x+ \cdots + a_kx^k.
$$
The Newton polygon of $g$ is defined to be the lower convex hull of the points
$$
(i,v_p(a_i)) \subset \R^2, \ \ 0 \leq i \leq k.
$$ 
Note that some authors reverse the indices. \\

\noindent \textbf{Acknowledgments.} We would like to thank Farshid Hajir for helpful conversations and for suggesting that we try to interpret the Pad\'e approximants in terms of the Generalized Laguerre Polynomials.

\section{Results}

We start by recalling the well-known discriminant formula \cite[(1.1)]{hajir}, originally due to Schur: 
\begin{align} \label{glp_disc}
\disc (-1)^nn!L_n^{(\alpha)}(x) = \prod_{j=2}^{n} j^j (\alpha +j)^{j-1},
\end{align}
from which we immediately obtain a discriminant formula for each of the approximants.

\begin{prop} \label{discprop}
Let $(u,v) \in \NN \times \NN$ and let 
$$
\Delta(u,v) = (-1)^u \prod_{j=1}^u j^j \prod_{j=1}^{u-1}(v+j)^{u-j}.  
$$
Then $\disc P(u,v,x) = \Delta (u,v)$ and $\disc Q(u,v,x)= \Delta(v,u)$.
\end{prop}

We also recall a criterion of Jordan for showing that a transitive subgroup of $S_n$ contains $A_n$.  This is the technique employed in \cite{coleman} for showing the Galois groups of the $e_{n}(x)$ contain $A_n$, and similarly for the GLP in \cite{hajir}.  We will be content  simply to quote the relevant theorems we will need for our analysis.  

\begin{thm}[Jordan] \label{jordan}
Let $n \geq 7$ be a positive integer and suppose $G$ is a transitive subgroup of $S_n$.  If $G$ contains a $p$-cycle, for $p$ prime in the range $n/2 < p < n-2$, then $G$ contains $A_n$. 
\end{thm}

\begin{defn}[Definition 5.1 of \cite{hajir}]
Given $f \in \Q[x]$, let $N_f$, called the \emph{Newton Index of $f$}, be the least common multiple of the denominators (in lowest terms) of all slopes of $\NP_p( f )$ as $p$ ranges over all primes.
\end{defn}

The main theorem, which we quote from \cite[Theorem NP]{coleman}, of Newton Polygons is given by the following.  

\begin{thm}[Main Theorem of Newton Polygons]
Let $p$ be a prime.  Let $(x_0,y_0)$, $(x_1,y_1)$, \dots, $(x_p,y_p)$ denote the vertices of the Newton Polygon of a polynomial $g\in \Q[x]$.  Then over $\Q_p$, $g$ factors as follows:
$$
g_1(x) g_2(x) \cdots g_r(x)
$$
where the degree of $g_i$ is $x_i-x_{i-1}$ and all the roots of $g_i(x)$ in $\overline{\Q}_p$ have valuation $-\left(\frac{y_i-y_{i-1}}{x_i-x_{i-1}} \right)$.
\end{thm}

An important consequence is that the denominators of the slopes are ramification indices, and hence divide the order of the Galois group.  By incorporating Jordan's theorem, we get a powerful technique for showing the Galois group is large.

\begin{thm}[Theorem 4.3 of \cite{hajir}] \label{hajirjordan}
Given an irreducible polynomial $f \in \Q[x]$, $N_f$ divides the order of the Galois group of $f$. Moreover, if $N_f$ has a prime divisor $q$ in the range $n/2 < q < n-2$, where $n$ is the degree of $f$, then the Galois group of $f$ contains $A_n$.
\end{thm}

We now apply this technique to the Pad\'e approximants, and prove Theorem \ref{diagthm} in two steps.  The first is given by the following proposition.

\begin{prop} \label{diaggal}
Let $m \in \NN$ and assume the polynomials $P(m,m+1,x)$ and $Q(m,m+1,x)$ are irreducible over $\Q$.  Then $\Gal_{\Q} P(m,m+1,x) \supseteq A_{m}$, and $\Gal_{\Q} Q(m,m+1,x) \supseteq A_{m+1}$.
\end{prop}

We only provide the details for $P(m,m+1,x)$ since those for $Q(m,m+1,x)$ are nearly identical.  By assumption, the polynomial $P(m,m+1,x)$ is irreducible over $\Q$ of degree $m$, hence its Galois group is a transitive subgroup of $S_m$. We will show that the Newton index contains a prime $p$ in the range $m/2<p<m-2$ and apply Theorem \ref{hajirjordan}.  We start with a preliminary lemma.

\begin{lem}\label{bert23}
Let $m\geq 21$.  Then the interval $(2m/3,m-3)$ contains a prime number.
\end{lem}

\begin{proof}
Let $\pi: \R \to \NN$ be the prime counting function.  By \cite[Corollary 1]{rs}, once $x \geq17$ we have 
$$
\frac{x}{\log x} < \pi(x) < 1.25506\, \frac{x}{\log x}. 
$$
Thus, we can ensure that $(2m/3,m-3)$ contains a prime number by taking $m$ large enough so that
$$
\frac{m-3}{\log(m-3)} - 1.25506\,\frac{2m/3}{\log(2m/3)} \geq 1;
$$
it suffices to take $m \geq 95$.  A finite search among $m \leq 95$ shows that $\pi(m-4) - \pi(2m/3) \geq 1$ once $m \geq 21$.
\end{proof}

\begin{proof}[Proof of Proposition \ref{diaggal}]
By Lemma \ref{bert23},  the interval $(2m/3,m-3)$ contains a prime number once $m \geq 21$.  Fix such a prime $p$ and write 
$$
p = m-3 - k,
$$
for some $1 \leq k \leq \lfloor \frac{m}{3} \rfloor -4$.  We write $P(m,m+1,x) = \sum_{j=0}^m a_jx^j$ and recall that 
$$
a_0 = (2m+1)(2m) \cdots (m+2).
$$
Since $p<m$, we have that $v_p(a_0) \geq 1$.  But since $p>2m/3$, we have that $v_p(a_0) \leq 1$, hence $v_p(a_0) = 1$.

Next, observe that for the indices
$$
j=p, p+1,\dots,m,
$$
we have $v_p(a_j) = 0$, since neither $(2m+1-j)!/(m+1)!$ nor $\binom{m}{j}$ is divisible by $p$.   For the intermediate coefficients $a_j$ with $1 \leq j \leq p-1$, we have 
$$
a_j = (2m+1-j)(2m-j)\cdots(m+2) \binom{m}{j},
$$
and $v_p \binom{m}{j} = 1$ for each $1 \leq j \leq p-1$.  Thus $v_p(a_j) \geq 1$ for each $1 \leq j \leq p-1$.

Therefore the $p$-adic Newton Polygon of $P(m,m,x)$  is exactly the one pictured in Figure \ref{fig1}, and we conclude that $\Gal_\Q(P(m,m+1,x))$ contains $A_m$.
\begin{figure}[htbp] \label{fig1}
\begin{center}
\begin{tikzpicture}[scale=1.5,sizefont/.style={scale = 2}]
        \draw [<->,thick] (0,2) node (yaxis) [above] {}
        |- (9.5,0) node (xaxis) [right] {};
   \draw[ultra thick] (0,1.5)   --  (7,0); 
   \draw[ultra thick] (7,0)   --  (9,0); 
   \draw[dotted] (0,1.5)   --  (6.9,1.5);
   \draw (0,1.5) node {$\bullet$};
   \draw (0,1.5) node[left] {$1$};
  \draw (7,0) node {$\bullet$};
  \draw (7,0) node[below] {$p$};
 \draw (9,0) node[below] {$m$};
  \draw (9,0) node {$\bullet$};  
  \draw (4,1) node {{\tiny slope = $-1/p$}};
   
\end{tikzpicture}
\end{center}
\caption{$p$-adic Newton Polygon of $P(m,m+1)(x)$ and $Q(m,m+1)(x)$, for $m \geq 21$}
\end{figure}
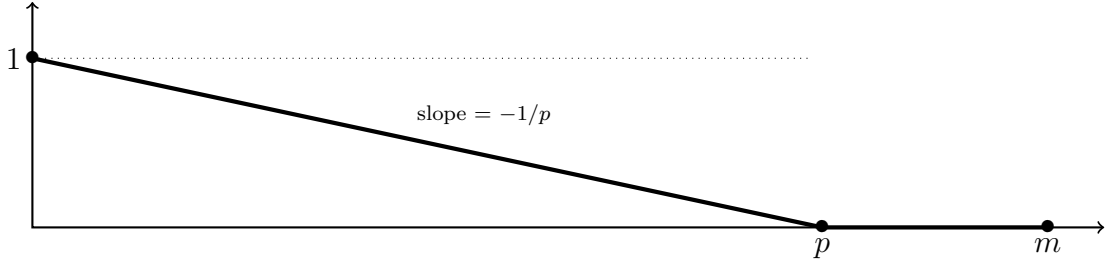
For degrees $1 \leq m \leq 20$, we easily check in \textsf{Magma} that  $\Gal_\Q(P(m,m+1,x)) \supseteq A_m$.  This completes the proof.
\end{proof}

\begin{prop}\label{diagdisc}
Let $m>1$.  Then 
\begin{enumerate}
\item $\disc P(m,m+1,x) \in \Q^{\times 2}$ if and only if $m\equiv 0\pmod{4}$ or if $(m+1)/2=(2r+1)^2$ for some $r\geq0$, \label{diagdisc2}
\item $\disc Q(m,m+1,x) \in \Q^{\times 2}$  if and only if $m+1=(2r+1)^2$ for some $r\geq0$. \label{diagdisc3}
\end{enumerate}
\end{prop}

\begin{proof}
For (\ref{diagdisc2}), observe that $\disc P(m,m+1,x)<0$ if $m \equiv 2,3 \pmod{4}$.  If $m \equiv 0 \pmod{4}$, then 
$$
\disc P(m,m+1,x) = \prod_{j=1}^{m}j^j\prod_{j=1}^{m-1}(m+1+j)^{m-j}.
$$
We will show $\disc P(m,m+1,x)$ is trivial in the group $\Q^{\times}/\Q^{\times 2}$.  Writing $m=4k$, we calculate directly 
\begin{align*}
\prod_{j=1}^{m}j^j\prod_{j=1}^{m-1}(m+1+j)^{m-j} &\equiv 1\cdot 3 \cdot 5 \cdots (4k-1) \times (4k+2)(4k+4) \cdots 8k \\
&\equiv \frac{(4k)!}{2\cdot 4 \cdot 6 \cdots 4k} \times\frac{2^{2k} (4k)!}{(2k)!} \\
&\equiv \frac{(4k)!}{2^{2k}(2k)!} \times\frac{2^{2k} (4k)!}{(2k)!} \pmod{\Q^{\times 2}}, 
\end{align*}
which is trivial in $\Q^{\times}/\Q^{\times 2}$. 

Next, suppose $m \equiv 1\pmod{4}$ and write $m=4r+1$.  Working modulo rational squares, we will show that $\disc P(m,m+1,x) \in \Q^{\times 2}$ if and only if $2r+1 \in \Q^{\times 2}$.  We compute:
\begin{align*}
\disc P(m,m+1,x) &\equiv 1\cdot 3 \cdot 5 \cdots (4r+1) \times (4r+4)(4r+6) \cdots (8r+2) \\
&\equiv 1\cdot 3 \cdot 5 \cdots (4r+1) \times 2^{2r} (2r+2)(2r+3)\cdots (4r+1)\\
&\equiv 1\cdot 3 \cdot 5 \cdots (4r+1) \times (2r+2)(2r+4)\cdots (4r) \times (2r+3)(2r+5)\cdots (4r+1)\\
&\equiv 1\cdot 3 \cdot 5 \cdots (2r+1) \times 2^r(r+1)(r+2)\cdots (2r)\\
&\equiv \frac{(2r+1)!}{2\cdot 4 \cdots (2r)} \times \frac{2^r (2r)!}{r!}\\
&\equiv 2r+1 \pmod{\Q^{\times 2}},
\end{align*}
from which the statement follows.

For (\ref{diagdisc3}) we have 
$$
\disc Q(m,m+1,x) = (-1)^{\binom{m+1}{2}}\prod_{j=1}^{m+1}j^j \prod_{j=1}^{m}(m+j)^{m+1-j},
$$
and similar arguments to the one above show
$$
\disc Q(m,m+1,x) \equiv
\begin{cases}
2 \pmod{\Q^\times/\Q^{\times 2}} & \text{if $m \equiv 3 \pmod 4$}\\
m+1 \pmod{\Q^\times/\Q^{\times 2}} &  \text{if $m \equiv 0 \pmod 4$}.
\end{cases}
$$
If $m \equiv 1,2 \pmod{4}$, then $\disc Q(m,m+1,x)<0$.  This completes the proof.
\end{proof}

Together with the results of Grosswald, this determines the Galois group of all the diagonal approximants.  This leaves open the question of irreducibility.  While we do not investigate the general irreducibility of the polynomials $P(m,m+1,x)$ and $Q(m,m+1,x)$ here, we expect them to be amenable to a similar analysis as in \cite{FT}.  Instead, we work out some special cases of their Eisenstein and ``near-Eisenstein'' properties.

\subsection{Factorizations mod $p$} Suppose $\lbrace f_n(x) \rbrace_{n=1}^\infty$ is a sequence of integral polynomials such that $\deg f_n(x) = n$ and let $p$ be a prime number.  Write $n$ in base $p$ as
$$
n = a_0 + a_1p +a_2p^2 + \cdots + a_rp^r.
$$
We say that  $\lbrace f_n(x) \rbrace_{n=1}^\infty$ \emph{admits a Schur factorization mod $p$} if
$$
f_n(x) \equiv f_{a_0}(x)f_{a_1}(x)^p\cdots f_{a_r}(x)^{p^r} \pmod{p}.
$$
If $\lbrace f_n(x) \rbrace_{n=1}^\infty$ admits a Schur factorization mod $p$ for all primes $p$, then we say that $\lbrace f_n(x) \rbrace_{n=1}^\infty$ \emph{has the Schur factorization property}.  Some important families of orthogonal polynomials, such as the Legendre, have the Schur factorization property \cite{carlitz}, while others, such as the Generalized Laguerre Polynomials, satisfy weaker, but still consequential, factorizations.  In particular, the GLP satisfy \cite[(4.8)]{carlitz}
$$
L^{(\alpha)}_{n+m}(x) \equiv (-x)^mL_{n}^{(\alpha)}(x) \pmod{m}
$$
whenever $\alpha$ is integral mod $m$. As regards the diagonal approximants, if we write the degree of the polynomial as 
$$
u + kp,
$$
for a prime number $p$, a non-negative integer $k$, and $u \in [0,p-1]$, then we get
\begin{align*}
P(m,m+1,x) &\equiv x^{kp}P(u,u+1,x) \pmod{p}\\
P(m,m,x) &\equiv x^{kp}P(u,u,x) \pmod{p}\\
Q(m,m+1,x) &\equiv x^{kp}Q(u,u+1,x) \pmod{p}\\
Q(m,m,x) &\equiv x^{kp}Q(u,u,x) \pmod{p}.
\end{align*}
It would be interesting to understand bases for the rings of integers in the root fields defined by the diagonal approximants, since the above factorizations would give an insight into the wild ramification in those fields.  We do not pursue that question here, but rather utilize these factorizations below in the ``near-Eisenstein'' case.

\subsection{Special Cases of Irreducibility}  Irreducibility results for parametric families of polynomials $\lambda_n^{(t)}(x) \in \Q(t)[x]$ tend to come in two flavors: ``fix $n$, vary $t$'', and ``fix $t$, vary $n$''.  In both cases, the GLP have been shown to be irreducible outside a finite set, though that finite set may not be effectively computable in some cases (see \cite{hw}).  In contrast, there are classical results of Holt \cite{holt1, holt2} and Wahab \cite{wahab1, wahab2} which give the irreducibility of the Legendre polynomials when the degree is a prime power, multiple of a prime, or within a few units of such.  We conclude this paper with two observations on specific cases of irreducibility for the diagonal approximants that are reminiscent of these.  First, we recall the Eisenstein-Dumas criterion for irreducibility \cite{dumas}:

\begin{thm}
Let $A$ be a unique factorization domain and 
$$
f(x) = \sum_{j=0}^n a_j x^j \in A[x]
$$
with $a_0a_n \ne 0$. If the Newton polygon of $f$ with respect to some prime $\p \in A$ consists of a single segment from $(0,m)$ to $(n,0)$ with $\gcd(m,n)=1$, then $f$ is irreducible in $A[x]$.
\end{thm}

\begin{thm} \label{eisenstein}
Let $p$ be an odd prime number and $n \geq 1$.  Then the polynomials $P(p^n,p^n+1,x)$, $Q(p^n-1,p^n,x)$, $P(p^n,p^n,x)$, and $Q(p^n,p^n,x)$ are irreducible over $\Q$. 
\end{thm} 

The irreducibility of $P(p^n,p^n,x)$ and $Q(p^n,p^n,x)$ is already known  by \cite{FT} since they are Bessel Polynomials.  For the polynomials $P(p^n,p^n+1,x)$ and $Q(p^n-1,p^n,x)$, we will show that they are Eisenstein-Dumas at $p$.  We will give all details of the proof for $P(p^n,p^n+1,x)$ and leave the case of $Q(p^n-1,p^n,x)$ as an exercise for the reader -- the calculations are nearly identical. However, because $P(u,v,x) = Q(v,u,-x)$, we immediately get the following corollary.

\begin{cor}
Let $p$ be an odd prime number and $n \geq 1$.  Then the polynomials $P(p^n,p^n-1,x)$, $Q(p^n,p^n+1,x)$ are irreducible over $\Q$. 
\end{cor}

 Our strategy of proof boils down to two steps, since $P(p^n,p^n+1,x)$ is monic. First, we will show that $v_p(a_0)$ is coprime to $p$ and second we will show that the $p$-adic Newton Polygon consists of a single segment.  For notational ease, we break the proof of Theorem \ref{eisenstein} into several lemmas and we begin by setting some notation.

Let $p$ be an odd prime, fix $n \geq 1$, and write $P(p^n,p^n+1,x) = \sum_{j=0}^{p^n} a_jx^j$.  

\begin{lem}
With all notation as above, $v_p(a_0)$ is coprime to $p$.
\end{lem}

\begin{proof}
We write
\begin{align*}
v_p(a_0) &= v_p \left( \frac{(2p^n+1)!}{(p^n+1)!} \right) = v_p ((2p^n+1)!) - v_p ((p^n+1)!)
\end{align*}
and appeal to the well-known fact that if we write $m$ in base $p$ as
$$
m = \sum_{j=0}^rb_jp^j,
$$
then 
$$
v_p(m!) = \frac{m - (\sum_{j=0}^rb_j)}{p-1}.
$$
Applied to $(2p^n+1)!$ and $(p^n+1)!$, we get
\begin{align*}
v_p(a_0) &= \frac{2p^n +1 - (1+2)}{p-1} - \frac{p^n+1-(1+ 1)}{p-1} = \frac{p^n- 1}{p-1},
\end{align*}
which is coprime to $p$.
\end{proof}

\begin{lem} \label{ajlem}
With all notation as above, we have 
$$
v_p(a_j) = v_p(a_0) + v_p(j-1) - v_p(j!), 
$$
for $j=1,\dots,p^n$.
\end{lem}

\begin{proof}
Fix an index $j$, and write $j$ and $2p^n+1-j$ in base $p$ as
\begin{align*}
j &= \sum_{k=0}^{n} c_kp^k \\
2p^n+1-j &= \sum_{k=0}^{n} u_kp^k,
\end{align*}
respectively.  Since
$$
a_j = \frac{(2p^n+1-j)!}{(p^n+1)!} \binom{p^n}{j},
$$
we have
\begin{align} \label{vpaj}
v_p(a_j) &= v_p \left( \frac{(2p^n+1- j)!}{(p^n+1)!} \right) + v_p \binom{p^n}{j}  = \frac{2p^n +1 -j - \left[\sum_{k} u_k\right]}{p-1} - \frac{p^n-1}{p-1} + n - v_p(j). 
\end{align}
To finish the proof, we break our argument into three cases.  In each one, we deduce a formula for $\sum_{k }u_k$ which allows us to conclude $v_p(a_j) = v_p(a_0) + v_p(j-1) - v_p(j!)$. \\

\noindent \textbf{Case 1: $p \equiv 0 \pmod{p}$}.  Then we write 
\begin{align*}
j &=\sum_{k=s}^{n-1} c_kp^k \\
2p^n + 1 - j &= \sum_{k=0}^n u_kp^k = 1 + \left(\sum_{k=s}^{n-1} u_{k}p^k \right) + p^n,
\end{align*}
where 
\begin{align*}
u_s &= p-c_s,\\
u_{k} &=p-1-c_k,\text{ for $k=s+1,\dots,n-1$};
\end{align*}
note that $v_p(j) = s$.  Then 
$$
\sum_{k=0}^n u_k = 1 + (p-c_s) + \sum_{k=s+1}^{n-1} (p-1-c_k)  +1 = 3+(n-s)(p-1) - \sum_{k=s}^{n-1}c_k.
$$
Substituting this into (\ref{vpaj}) gives
\begin{align*}
v_p(a_j) &=  \frac{2p^n +1 -j - \left[3+(n-s)(p-1) - \sum_{k=s}^{n-1}c_k\right]}{p-1} - \frac{p^n-1}{p-1} + n - v_p(j) \\
&=\frac{p^n-1}{p-1} -n + s - \left( \frac{j - \sum_k c_k}{p-1} \right) +n - s = v_p(a_0) - v_p(j!).
\end{align*}
Since $v_p(j-1) = 0$, this agrees with the statement of the Lemma. \\

\noindent \textbf{Case 2: $j \equiv 1 \pmod{p}$.}  We proceed similarly and write 
\begin{align*}
j &=1+ \sum_{k=s}^{n-1} c_kp^k \\
2p^n + 1 - j &= \sum_{k=0}^n u_kp^k = \left(\sum_{k=s}^{n-1} u_{k}p^k \right) + p^n,
\end{align*}
where 
\begin{align*}
u_s &= p-c_s,\\
u_{k} &=p-1-c_k,\text{ for $k=s+1,\dots,n-1$};
\end{align*}
note that $v_p(j-1) = s$ and $v_p(j) = 0$.  Then
$$
\sum_{k=s}^n u_k = (p-c_s) + \sum_{k=s+1}^{n-1} (p-1-c_k) + 1 = 1 + (n-s)(p-1) + 1 - \sum_{k=s}^{n-1} c_k, 
$$
and so 
\begin{align*}
v_p(a_j) &=  \frac{2p^n +1 -j - \left[ 1 + (n-s)(p-1) + 1 - \sum_{k=s}^{n-1} c_k  \right]}{p-1} - \frac{p^n-1}{p-1} + n - v_p(j) \\
&= \frac{p^n -1}{p-1} -n + s - \left(\frac{j - 1 - \sum_{k=s}^{n-1} c_k}{p-1} \right) +n = v_p(a_0) + v_p(j-1) -v_p(j!).
\end{align*}

\noindent \textbf{Case 3: $j \equiv c_0 \pmod{p}$, $c_0 \ne 0,1$.}  This is nearly identical to the other cases with the restriction that $c_0 \not \in \lbrace 0, 1 \rbrace$ and with $u_k$ given by
\begin{align*}
u_0 &= p+1-c_0,\\
u_{k} &=p-1-c_k,\text{ for $k=s+1,\dots,n-1$}\\
u_n &= 1.
\end{align*}
Observe that $v_p(j) = v_p(j-1) = 0$.  Following the same approach as in the previous cases we first evaluate
$$
\sum_{k=0}^n u_k = p+1-c_0 + \sum_{k=0}^{n-1} (p-1 - c_k)  +1 = 3 + n(p-1) - \sum_{k=0}^{n-1}c_k
$$
and then compute 
\begin{align*}
v_p(a_j) &=  \frac{2p^n +1 -j - \left[  3 + n(p-1) - \sum_k c_k   \right]}{p-1} - \frac{p^n-1}{p-1} + n - v_p(j) \\
&=\frac{p^n -1}{p-1} -n - \left(\frac{j - 1 - \sum_{k=s}^{n-1} c_k}{p-1} \right) +n = v_p(a_0)  -v_p(j!).
\end{align*}
This completes the proof.
\end{proof}

\begin{lem}
The $p$-adic Newton Polygon of $P(p^n,p^n+1,x)$ consists of a single segment.
\end{lem}

\begin{proof}
Specifically, we will show that the Newton Polygon is exactly 
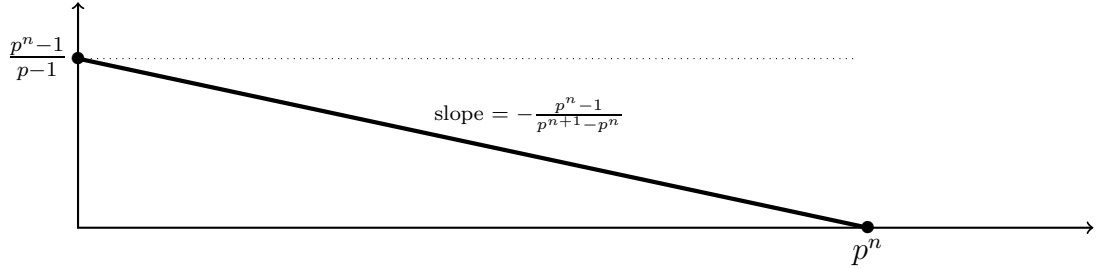
\begin{center}
\begin{tikzpicture}[scale=1.5,sizefont/.style={scale = 2}]
        \draw [<->,thick] (0,2) node (yaxis) [above] {}
        |- (9,0) node (xaxis) [right] {};
   \draw[ultra thick] (0,1.5)   --  (7,0); 
   \draw[dotted] (0,1.5)   --  (6.9,1.5);
   \draw (0,1.5) node {$\bullet$};
   \draw (0,1.5) node[left] {$\frac{p^n-1}{p-1}$};
  \draw (7,0) node {$\bullet$};
  \draw (7,0) node[below] {$p^n$};
  \draw (4,1) node {{\tiny slope = $-\frac{p^n-1}{p^{n+1}-p^n}$}};
   
\end{tikzpicture}
\end{center}
by showing that $v_p(a_j)$ lies above the indicated line, \emph{i.e.}~that
\begin{align}\label{ineq}
v_p(a_j) \geq \frac{(p^n-1)\left(1 - \frac{j}{p^n}\right)}{p-1}
\end{align}
for $j=1,\dots,p^n-1$.  By Lemma \ref{ajlem}
$$
v_p(a_j) = \frac{p^n-1- j + \sum_kc_k}{p-1} + v_p(j-1),
$$
so (\ref{ineq}) is equivalent to showing 
\begin{align} \label{ineq2}
(p-1)v_p(j-1) + \sum_k c_k \geq \frac{j}{p^n}.
\end{align}
This is clear because the sum on the left is at least 1, and the fraction on the right is at most $\frac{p^n-1}{p^n}$.
\end{proof}

With Theorem \ref{eisenstein} proved, we can give good evidence for irreducibility when the degree is within 1 of a prime as well.  To do this, we recall the notion of \emph{flatness} and \emph{steepness} of a Newton Polygon \cite[Defn.~2.1]{bh} as well as the main result of \cite{bh}.  We import their notation: let $K$ be a field eqipped with a normalized discrete valuation $v: K \to \Z$ and let $K_v$ be the completion of $K$ with respect to $v$; for $f \in K_v[x]$, let $\NP_v(f)$ denote the Newton Polygon of $f$ with respect to $v$.

\begin{defn}[Definition 2.1 of \cite{bh}]
For $f \in K_v[x]$, let $\lambda_v(f) \geq 0$ be the length of the slope 0 segment of $\NP_v(f)$; if none of the slopes is 0, $\lambda_v(f) = 0$. We call $\lambda_v(f)$ the \emph{flatness} of $f$ with respect to $v$. We define $\mu_v(f) = \max_{1\leq i \leq r} |m_i|$ where $m_1,\dots,m_r$ are the slopes of $\NP_v(f)$ and call it the \emph{steepness} of $f$ with respect to $v$. We say that $\NP_v(f)$ is non-trivial if $\mu_v(f)\ne 0$. Define
$$
\ell_v(f) \ddef \lbrace i \in [0,n] \cap \Z : v(a_i) = \min_{0 \leq j \leq n} v(a_j)\rbrace. 
$$
\end{defn}

\begin{defn}[Definition 2.3 of \cite{bh}]
A polynomial $\tilde{f}(x) = \sum_{i=0}^n \tilde{a}_i x^i \in K_v[x]$ is called a $v$-admissible deformation of $f(x) =\sum_{i=0}^n a_ix^i$ if
\begin{itemize}
\item $v(\tilde{a}_i) \geq v(a_i)$ for all $0 \leq i \leq n$, 
\item for $i=0$, for $i=n$, and for some integer $i \in \ell_v(f)$, $v(\tilde{a}_i) = v(a_i)$.
\end{itemize}
\end{defn}

\begin{thm}[Corollary 2.7 of \cite{bh}] \label{bhthm}
Suppose $f(x) =\sum_{i=0}^n a_ix^i \in K_v[x]$ is $v$-monic and $v$-integral, and satisfies $v(a_j) > 0$ for $0 \leq j \leq n - u - 1$. If the left-most slope of $\NP_v(f)$ is $-\mu$, then no $v$-admissible deformation of $f$ admits a factor in $K_v [x]$ with degree belonging to $(u, \mu^{-1} )$.
\end{thm}

Roughly speaking, polynomials that are ``nearly'' $v$-Eisenstein have a restricted range of degrees of their irreducible factors. We now apply Theorem \ref{bhthm} to a special case of the diagonal approximants.

\begin{prop}
Let $p$ be an odd prime number.  Then 
\begin{enumerate}
\item the polynomials $P(p+1,p+2,x)$, $Q(p,p+1,x)$ are either irreducible over $\Q$, or factor into a linear and an irreducible degree-$p$ factor, and 
\item  the latter case occurs if and only if there exists an integer $m$ such that
\begin{align*}
P(p+1,p+2,-3 + pm) &=0, \\
Q(p,p+1,-1+pm) &= 0,
\end{align*}
respectively.
\end{enumerate}
\end{prop}

\begin{proof}
We focus on the case of $P(p+1,p+2,x)$, with the details for $Q(p,p+1,x)$ being similar.  For (1), we observe that  $P(p+1,p+2,x)$ is integral and monic, and that the quantities ``$u$'' and ``$\mu$'' of Theorem \ref{bhthm} are given by 
$$
u=1,\  \mu = 1/p.
$$
By Theorem \ref{bhthm}, $P(p+1,p+2,x)$ does not have a $\Q_p$-factor with degree belonging to $(1,p)$.  Hence, either $P(p+1,p+2,x)$ is irreducible over $\Q_p$, or factors into a linear and irreducible degree-$p$ factor over $\Q_p$ (and hence over $\Q$).

For (2), factoring $P(p+1,p+2,x)$ mod $p$ gives 
$$
P(p+1,p+2,x) \equiv x^{p}P(1,2,x) \equiv x^{p}(x+3) \pmod{p}.
$$ 
By Hensel's Lemma,  $P(p+1,p+2,x)$ has a $\Q_p$-root $\alpha$, such that $\alpha \equiv -2 \pmod{p}$.
By Gauss' Lemma, since $P(p+1,p+2,x)$ is monic and integral, if it has a rational root, it has an integral root.  

Combining these observations, we see that either $P(p+1,p+1,x)$ is irreducible over $\Q$, or factors over $\Q$ into a linear and an irreducible degree $p$ factor, with the root of the linear factor $\equiv -2 \pmod{p}$.
\end{proof}

\begin{rmk}
We remark that for $n >1$ the polynomials $P(np+1,np+2,x)$ and $Q(np,np+1,x)$ have interesting Newton Polygons from the point of view of Theorem \ref{bhthm}, but with more possibilities for factorizations due to the more complicated value of $\mu$. 
\end{rmk}

\end{document}